\documentclass{article}
\usepackage{amsmath,amsfonts,amsthm,amssymb,mathrsfs,latexsym,paralist, xypic}

\usepackage{tikz}
\usetikzlibrary{arrows,automata}

\tikzstyle{V}=[fill=black,circle,scale=0.4, outer sep = 4pt]

\newtheorem{thm}{Theorem}[section]
\newtheorem{prop}[thm]{Proposition}
\newtheorem{cor}[thm]{Corollary}
\newtheorem{lemma}[thm]{Lemma}
\theoremstyle{remark}

\newtheorem{example}[thm]{Example}
\theoremstyle{definition}
\newtheorem{defn}[thm]{Definition}

\DeclareMathOperator{\supp}{supp}

\renewcommand{\2}{^{(2)}}
\newcommand{\inv}{^{-1}}

\newcommand{\bi}{\begin{itemize}}
\newcommand{\ei}{\end{itemize}}
\newcommand{\be}{\begin{enumerate}}
\newcommand{\ee}{\end{enumerate}}

\newcommand{\T}{\mathbb{T}}
\renewcommand{\H}{\mathcal{H}}

\newcommand{\G}{\mathcal{G}}

\newcommand{\R}{\mathbb{R}}
\newcommand{\N}{\mathbb{N}}
\newcommand{\Z}{\mathbb{Z}}

\begin{document}

\title{$K$-theory and Homotopies of 2-cocycles on Group Bundles}
\author{Elizabeth Gillaspy}
\date{\today}
\maketitle

\begin{abstract}
This paper continues the author's program to investigate the question of when a homotopy of 2-cocycles $\Omega = \{\omega_t\}_{t \in [0,1]}$ on a locally compact Hausdorff groupoid $\G$ induces an isomorphism of the $K$-theory groups of the  twisted groupoid $C^*$-algebras:
\[K_*(C^*(\G, \omega_0)) \cong K_*(C^*(\G, \omega_1)).\]
Building on our earlier work in \cite{transf-gps, eag-kgraph}, we show that if $\pi: \G \to M$ is a locally trivial bundle of amenable groups over a locally compact Hausdorff space $M$, a homotopy $\Omega = \{\omega_t\}_{t \in [0,1]}$ of 2-cocycles on $\G $ gives rise to an isomorphism 
\[K_*(C^*(\G, \omega_0)) \cong K_*(C^*(\G, \omega_1)).\]


\end{abstract}

\section{Introduction}
Let $\pi: V \to M$ be a real or complex vector bundle over a manifold $M$.  A bilinear 2-form $\sigma: V\2 \to \R$ induces a \emph{homotopy of 2-cocycles} $\{\omega_t\}_{t \in[0,1]}$ on $V$: If $\pi(v) = \pi(w)$, define 
\[\omega_t(v,w) := e^{2\pi i t \sigma(v,w)}.\]
Plymen proved in Theorem 1 of \cite{plymen-weyl-bdl} that when $V$ is an even-dimensional real vector bundle and $\sigma$ is a symplectic 2-form on $V$, the twisted $C^*$-algebra of the vector bundle $C^*(V, \omega_1)$ is a continuous-trace $C^*$-algebra  over $M$, with trivial Dixmier-Douady class, and hence is Morita equivalent to $C_0(M)$.  

Furthermore, applying a fiberwise Fourier transform to $V$, one sees immediately that
\[C^*(V, \omega_0) = C^*(V) \cong C_0(V^*),\]
where $V^*$ is the dual bundle to $V$.  Since $V$ is even-dimensional, the Thom isomorphism in $K$-theory tells us that 
\[K_*(C_0(V^*)) \cong K_*(C_0(M));\]
consequently, Plymen's result implies that the homotopy of 2-cocycles $\{\omega_t\}_{t\in[0,1]}$ associated to a symplectic form $\sigma$ induces an isomorphism
\[K_*(C^*(V, \omega_0)) \cong K_*(C^*(V, \omega_1)).\]

In this article, we present a substantial generalization of this result:
\newtheorem*{main}{Corollary \ref{main-cor}}
\begin{main}
Let  $\{\omega_t\}_{t\in[0,1]}$ be a homotopy of 2-cocycles on a second countable, locally trivial, locally compact Hausdorff group bundle $\pi: \G \to M$, such that the fiber group $G = \pi\inv(m)$ is amenable. Then the homotopy induces an isomorphism
\[K_*(C^*(\G, \omega_0)) \cong K_*(C^*(\G, \omega_1))\]
of the $K$-theory groups of the twisted $C^*$-algebras.
\end{main}

While the motivation (and main applications) of this result arise from considering vector bundles over manifolds, the proofs are no simpler in this special case.  Consequently, we present the results here in their full generality.

Group bundles are examples of  groupoids; the results of this article thus continue the author's research program, begun in \cite{transf-gps,eag-kgraph}, to investigate the question of when a homotopy $\{\omega_t\}_{t\in[0,1]}$ of 2-cocycles on a groupoid $\G$ induces an isomorphism 
\begin{equation}
K_*(C^*(\G, \omega_0)) \cong K_*(C^*(\G, \omega_1))\label{kthy-isom}
\end{equation}
of the $K$-theory groups of the twisted groupoid $C^*$-algebras.  This question was inspired by the realization that Bott periodicity and the noncommutative tori can both be viewed as examples  of a $K$-theoretic isomorphism arising from a homotopy of 2-cocycles.  We hope that our investigation of the question of when a homotopy $\{\omega_t\}_{t\in[0,1]}$ of 2-cocycles on a groupoid $\G$ induces the $K$-theoretic isomorphism 
 \eqref{kthy-isom} will increase our understanding of the structure of (twisted) groupoid $C^*$-algebras, as well as shedding light on questions in $C^*$-algebraic classification and string theory.

The study of the full and reduced $C^*$-algebras $C^*(\G), C_r^*(\G)$ associated to a locally compact groupoid $\G$ was initiated by Jean Renault in \cite{renault}, and has since been pursued actively by many researchers.
  Although Renault also defined the twisted groupoid $C^*$-algebras $C^*(\G, \omega), C^*_r(\G, \omega)$ for a 2-cocycle $\omega \in Z^2(\G, \T)$ in \cite{renault}, these objects have received relatively little attention until recently. However, it has now become clear that twisted groupoid $C^*$-algebras can help answer many questions about  the structure of untwisted groupoid $C^*$-algebras (cf.~\cite{cts-trace-gpoid-II, cts-trace-gpoid-III, clark-aH, dd-fell, jon-astrid}), as well as classifying those $C^*$-algebras which admit diagonal subalgebras (also known as Cartan subalgebras) --- cf.~\cite{c*-diagonals}. 
    In another direction, \cite{TXLG} 
    shows how the $K$-theory of twisted groupoid $C^*$-algebras classifies $D$-brane charges in many flavors of string theory.

\subsection{Outline}
In addition to its philosophical links with \cite{transf-gps, eag-kgraph}, an attentive reader will notice similarities between the proofs presented in this article and several main results from \cite{transf-gps, eag-kgraph}. To be precise, we begin this article by  following the outline of the proof of Theorem 3.5 from \cite{eag-kgraph} to calculate the $C_0(M)$-algebra structure associated to  a locally trivial bundle of groups $\pi: \G \to M$.  Then we use the results of this calculation, together with  Theorem 5.1 from \cite{transf-gps} and a Mayer-Vietoris argument, to establish our main result in Theorem \ref{main}; Corollary \ref{main-cor} follows immediately.

{\bf Acknowledgments:} The author would like to thank Nigel Higson for helpful suggestions.  I am also indebted to my advisor, Erik van Erp, for his support during the writing of this article and throughout my graduate career.

\section{$C_0(M)$-algebra structure}
In this section, we describe the natural $C_0(M)$-algebra structure on $C^*(\G, \omega)$, where  $\pi: \G \to M$ is a locally trivial bundle of  groups over a locally compact Hausdorff space $M$, and $\omega$ is a 2-cocycle on $\G$.  In order to state our result more precisely, we begin with some definitions. 

We note that the following definition is non-standard in its requirement of local triviality; however, this hypothesis is necessary for the proofs of our later results, and is satisfied by our motivating example of a vector bundle.
\begin{defn}
An \emph{(amenable) group bundle} is a locally compact Hausdorff space
 $\G$ together with a continuous, open surjection $\pi: \G \to M$ onto a locally compact Hausdorff space $M$ such that $\G_m:= \pi\inv(m)$ is isomorphic to a fixed (amenable) group $G$ for every $m \in M$, and such that $\G$ is locally trivial: for every point $m \in M,$ there exists an open neighborhood $U$ of $m$ such that $\pi\inv(U)$ is homeomorphic to $ U \times G$.
\end{defn}

Given a group bundle $\pi: \G \to M$, write 
\[\G\2 = \{(x,y) \in \G \times \G: \pi(x) = \pi(y)\}.\]
  Note that $\G\2 \subseteq \G \times \G$ is a closed subspace; we equip it with the subspace topology. 
  
 Thanks to the isomorphism $\phi_m: \G_m \to G$, if $(x,y)\in \G\2$, then there is a unique element $z$ in $\G_m$ such that $\phi_m(z) =\phi_m(x)\phi_m(y)$. We will usually write $xy$ for this element.   Similarly, for each $x \in \G_m$ there is a unique element $x\inv \in \G_m$ such that  $x x\inv = x\inv x = \phi_m\inv(e)$. Moreover, the local triviality of $\G$   implies, in particular, that  $\phi_m: \G_m \to G$ is a homeomorphism for all $m \in M$.  As a consequence, the map $(x,y) \mapsto xy$ is continuous as a map $\G\2 \to \G$, and $x \mapsto x\inv$ is a continuous map from $\G$ to itself.  
 
Let $C_c(\G)$ denote the collection of those continuous complex-valued functions $f$  on the total space $\G$ of the group bundle such that $\supp f$ is compact, and let $\lambda$ be a fixed Haar measure on the fiber group $G \cong \G_m$ of the group bundle.   The local triviality of $\G$ then implies the following proposition:
 \begin{prop}
 \label{cts-int}
 Let $f \in C_c(\G)$. Then the function 
 \[m \mapsto \int_{\G_m} f(x) \, d\lambda(x)\]
 lies in $C_0(M)$.
 \end{prop}
\begin{proof}
We begin by observing that since Haar measure is, in particular, a Radon measure, and $\supp f$ is compact, we know that  $L_m:= \lambda(\G_m \cap \supp f)$ is finite for each $m \in M$.

Fix $m \in M$, and let $U$ be a neighborhood of $m$ such that $\G|_U \cong U \times G$.  Thanks to this isomorphism, we will write $f(n,g)$ rather than $f(x)$ for $x \in\G|_U$. 
Since $\supp f$ is compact, $\pi_2\left(\supp f \cap (U \times G)\right) \subseteq G$ is also compact; consequently, for any $\epsilon > 0$, we can find a smaller neighborhood $U_\epsilon \subseteq U$ of $m$ such that if $n \in U_\epsilon$, $|f(n, g) - f(m,g)| < \epsilon/L_m$ for all $g \in G$.  

  It follows that if $n \in U_\epsilon$, 
\[ \left| \int_{\G_n} f(n,g) d\lambda(g) - \int_{\G_m} f(m,g) d\lambda(g)\right| \leq \int_G |f(n, g) -f(m,g)| \, d\lambda(g) < \epsilon.\]
In other words,
\[m\mapsto \int_{\G_m} f(x) d\lambda(x)\in C_0(M)\] as claimed.
\end{proof}

\begin{defn} A \emph{2-cocycle} on $\G$ is a continuous circle-valued function $\omega: \G\2 \to \T$ such that whenever $(x,y),\,(y,z) \in \G\2$, the \emph{cocycle condition} holds:
\begin{equation}
\omega(xy, z) \omega(x,y) = \omega(x,yz) \omega(y,z).
\label{cocycle}
\end{equation}
\end{defn}

\begin{example}

\begin{enumerate}
\item For any group bundle $\G$, the function $\omega: \G\2 \to \T$ given by $\omega(x,y) = 1 \ \forall \ (x,y) \in \G\2$ is a 2-cocycle on $\G$, called the \emph{trivial 2-cocycle.}
\item As discussed in the Introduction, if $\pi: V\to M$ is a vector bundle and $\sigma: V\2 \to \R$ is a 2-form on $V$, then $\omega(v,w) := e^{2\pi i \sigma(v,w)}$ defines a 2-cocycle on $V$.
\end{enumerate}
\end{example}

A 2-cocycle on $\G$ allows us to define a twisted convolution multiplication on $\G$, which in turn will allow us to build the associated \emph{twisted $C^*$-algebra} $C^*(\G,\omega)$.  This is a particular case of the construction of a twisted groupoid $C^*$-algebra, as described in Chapter II of \cite{renault}.

 Given a 2-cocycle $\omega$ on $\G$ and $f, g \in C_c(\G)$, we define the twisted convolution product of $f$ and $g$ by 
\[f*g(x) := \int_{\G_{\pi(x)}} f(xy) g(y\inv) \omega(xy, y\inv) d\lambda(y).\]
  We also define an involution on $C_c(\G)$ that incorporates the 2-cocycle:
\[f^*(x) := \overline{f(x\inv) \omega(x, x\inv)}.\]
Renault checks in \cite{renault} Proposition II.1.1
that the multiplication is well defined (that is, that $f* g \in C_c(\G)$ as claimed) and associative, and that $(f^*)^* = f$ so that the involution is involutive.  (The proof of associativity relies on the cocycle condition \eqref{cocycle}.) 

Thus, we have a $*$-algebra structure on $C_c(\G)$.  To indicate the importance of the 2-cocycle in this structure, we will often write $C_c(\G, \omega)$ for this $*$-algebra.  To avoid confusion between the use of $*$ to indicate the multiplication and the involution, we will usually denote the multiplication in $C_c(\G, \omega)$ by juxtaposition: $fg := f* g$.

 The twisted $C^*$-algebra $C^*(\G, \omega)$ is the completion of $C_c(\G,\omega)$ with respect to the \emph{full} or \emph{universal} $C^*$-norm \eqref{full-norm}. In order to give the precise definition of the universal norm,  we require some preliminary definitions.

\begin{defn} Let $\H$ be a Hilbert space.  We say that a $*$-homomorphism $\pi: C_c(\G, \omega) \to B(\H)$ is a \emph{representation} of $C_c(\G, \omega)$ if it is nondegenerate in the sense that 
\[\overline{\text{span}}\{\pi(f)\xi: f \in C_c(\G, \omega), \xi \in \H \} = \H.\]
\end{defn}

\begin{defn} The \emph{$I$-norm} on $C_c(\G)$ is given by 
\[\| f\|_I =  \max\left\{ \sup_{m \in M} \int_{\G_m} | f(x)| \, d\lambda(x), \ \sup_{m \in M} \int_{\G_m} | f(x\inv)| \, d\lambda(x) \right\}.\]
We say that a representation $\phi$ of $C_c(\G)$ is \emph{$I$-norm-bounded} if $\| \phi(f) \| \leq \| f\|_I$ for any $f \in C_c(G)$.
\label{i-norm}
\end{defn}
Then, 
 Proposition II.1.11 of \cite{renault} combines with the comments following Definition II.1.5 in \cite{renault}
to tell us
 that 
\begin{equation}
\|f\| := \sup \{\|\phi(f)\|: \phi \text{ is an $I$-norm-bounded representation of } C_c(\G, \omega) \}
\label{full-norm}
\end{equation}
is a $C^*$-norm.
\begin{defn}[cf. \cite{renault} Definition II.1.12]
The \emph{(universal) twisted $C^*$-algebra of $\G$}, denoted $C^*(\G, \omega)$, is the completion of $C_c(\G, \omega)$ with respect to the norm \eqref{full-norm}.
\end{defn}

The goal of this section is Proposition \ref{c0-m}, in which we  prove that, despite this intricate definition of the norm on $C^*(\mathcal{G},\omega)$, this $C^*$-algebra admits a $C_0(M)$-algebra structure which makes it much more tractable.

\begin{defn}
Let $A$ be a $C^*$-algebra and $M$ a locally compact Hausdorff space.  We say that $A$ is a \emph{$C_0(M)$-algebra} if there exists a $*$-homomorphism 
\[\Phi: C_0(M) \to ZM(A).\]
Given $f \in C_0(M), a \in A$, we will usually write $f \cdot a$ for $\Phi(f) a$.
\end{defn}

A $C_0(M)$-algebra fibers over $M$ in a natural way.  If $m \in M$, let 
\[I_m := \overline{C_0(M\backslash \{m\}) \cdot A};\]
then $I_m$ is an ideal, and the quotient $A_m := A/I_m$ gives the fiber of $A$ at $m$.  Indeed, Theorem C.26 of \cite{xprod} tells us that there is a unique topology on the bundle $\mathcal{A} := \coprod_{m\in M} A_m$ such that $A = \Gamma_0(\mathcal{A})$ is the continuous sections of $\mathcal{A}$ that vanish at infinity.

In an analogous manner to the construction of $A_m$, for any $C_0(M)$-algebra $A$ and any closed subset $F \subseteq M$, we have a quotient $A_F$ of $A$:
\[A_F := A/I_F \text{ where } I_F := \overline{C_0(M \backslash F) \cdot A}.\]
Proposition \ref{c0-m} below describes the $C_0(M)$-algebra structure carried by the twisted $C^*$-algebra of a group bundle $\pi: \G \to M$.  A similar result is obtained by Goehle for crossed products by a group bundle in Proposition 1.2 and Lemma 1.4 of \cite{goehle-mmII}, and the proof of Proposition \ref{c0-m} below proceeds along similar lines to Goehle's proof, and also to the proof of Theorem 3.5 from \cite{eag-kgraph}.

\begin{prop}
Let $\pi: \G \to M$ be a group bundle and let $\omega$ be a 2-cocycle on  $\G$.  Then $C^*(\G, \omega)$ is a $C_0(M)$-algebra, with $C^*(\G , \omega)_F \cong C^*(\G|_F, \omega)$ for any closed $F \subseteq M$.
\label{c0-m}
\end{prop}
\begin{proof}
The action of $C_0(M)$ on $C^*(\G, \omega)$ is defined as one might expect: for $\phi \in C_0(M), \ f \in C_c(\G, \omega)$, define $\phi \cdot f \in C_c(\G, \omega)$ by 
\begin{equation}
\phi\cdot f(x)  = \phi(\pi(x)) f(x).
\label{Phi}
\end{equation}
It's immediate that the action is linear and multiplicative; because $\pi(x) = \pi(y)$ whenever $x,y$ are in the same fiber of the group bundle, and in particular we have $\pi(x) = \pi(x\inv)$, we also have 
\[f* (\phi \cdot g) = (\phi \cdot f) *g = \phi \cdot (f*g)\]
 for any $f,g \in C_c(\G, \omega)$ and any $\phi \in C_0(M)$.  In other words, $C_0(M)$ acts centrally on $C_c(\G, \omega)$.  Moreover, a straightforward check shows that $(\phi\cdot f)^* = \phi^*\cdot f^*,$ so the centrality of the action implies that it  is also $*$-preserving. 
 
Thus, to see that this action gives rise to a $*$-homomorphism $\Phi: C_0(M) \to ZM(C^*(\G, \omega))$, we merely need to check that the action is bounded.  That is, we will show that $\|\phi\cdot f\| \leq \|\phi\|_\infty \|f\|$ for any $f \in C_c(\G,\omega)$ and any $\phi \in C_0(M)$.  

Fix  $\phi \in C_0(M), f \in C_c(\G, \omega)$.  Letting $K_f = \text{supp } f$,  choose $\phi_f \in C_0(M)$ to be 1 on $\pi(K_f)$.  Then the function 
\[\xi(m) :=  (\|\phi\|_\infty^2 - |\phi(m)|^2) |\phi_f(m)|^2 \]
is in $C_0(M)$.   Moreover, $\xi$ is positive, and hence has a positive square root, $k$.  The positivity of $k$, combined with our earlier observations that the action is multiplicative, $*$-preserving, and central, means that 
\[ (\xi\cdot f)^*f = k^2 \cdot (f^*f) = (k \cdot f)^* (k \cdot f).\]
Therefore, $(\xi\cdot f)^*f \geq 0$ in $C^*(\G, \omega)$.  Since $\xi^* = \xi$, this inequality tells us that
\[ 0 \leq (\xi\cdot f)^*f = \xi \cdot( f^*f) = (\|\phi\|^2_\infty  |\phi_f|^2) \cdot (f^*f) - (|\phi \cdot \phi_f|^2)\cdot  (f^*f).\]
Since positivity preserves norms, it follows that 
\begin{equation}
 \| (|\phi \cdot \phi_f|^2)\cdot (f^*f)\| \leq \| (\|\phi\|_\infty^2 |\phi_f|^2) \cdot (f^*f)\|.
 \label{norm-compare}
 \end{equation}
 
 Observe that, since 
 $\phi_f = 1 $  on $\text{supp}\, f$,   the function $\|\phi\|_\infty^2 |\phi_f|^2$ acts on $f$ as multiplication by the constant $\|\phi\|_\infty^2$.  Moreover, 
\[ (|\phi  \cdot \phi_f|^2)\cdot f^*f = ((\phi\cdot  \phi_f) \cdot f )^*  \left((\phi \cdot \phi_f)\cdot f \right)= (\phi \cdot f)^* (\phi \cdot f)\]
because the action is multiplicative, central, and $*$-preserving.
Thus, Equation \eqref{norm-compare} becomes 
\[\|(\phi \cdot f)^* (\phi \cdot f)\| \leq \|\phi\|_\infty^2 \|f^*f\|,\]
so by the $C^*$-identity we have 
\[ \| \phi \cdot f \| \leq \|\phi\|_\infty \|f\|\]
as claimed.

 We have thus shown that the action is bounded, so it  extends to a $*$-homomorphism $\Phi: C_0(M) \to M(C^*(\G, \omega))$.  In fact, $\text{Im} \, \Phi \subseteq ZM(C^*(\G, \omega))$ because $\text{Im}\, \Phi$ acts centrally on the dense $*$-subalgebra $C_c(\G, \omega)$.
Moreover, $\Phi(\phi_f )f = f$ for any $f \in C_c(\G, \omega)$, so $\Phi(C_0(M))\cdot C^*(\G, \omega)$ contains the dense subalgebra $C_c(\G, \omega)$.  Consequently, $\overline{\Phi(C_0(M)) \cdot C^*(\G, \omega)} = C^*(\G, \omega)$. 
In other words, $\Phi$ makes $C^*(\G,\omega)$ into a $C_0(M)$-algebra as claimed.
We will use this action of $C_0(M)$ on $C^*(\G ,\omega)$ throughout the rest of this proof, usually denoting it by $\phi \cdot f$ as above rather than by $\Phi$.

Checking that $C^*(\G,\omega)_F \cong C^*(\G|_F , \omega)$ for any closed subset $F \subseteq M$ as claimed will require rather more work.

Recall that $C^*(\G,\omega)_F$ is given by the quotient $C^*(\G, \omega)/I_F$, where 
\[I_F:= \overline{\textrm{span}}\{\phi \cdot f: \phi \in C_0(M \backslash F), f \in C^*(\G,\omega)\}.\]
Thus, in order to prove the Proposition, we must show that 
\[C^*(\G|_F, \omega) \cong C^*(\G,\omega)/I_F.\]  We will begin by showing
 that we can indeed exhibit $C^*(\G|_F, \omega)$ as a quotient of $C^*(\G,\omega)$ whenever $F \subseteq M$ is closed.

Fix a closed subset $F \subseteq M$ and let $q_F: C_c(\G, \omega) \to C_c(\G|_{F}, \omega)$ be the restriction map.  By the definition of the $I$-norm given in Definition \ref{i-norm}, $q_F$ is $I$-norm-bounded; since the operations in the $*$-algebra $C_c(\G, \omega)$ respect the way $\G$ fibers over $M$, $q_F$ is also a $*$-homomorphism.  Consequently, for any $I$-norm-bounded representation $\psi$ of $C_c(\G|_{F}, \omega)$, the composition $\psi\circ q_F$ is an $I$-norm-bounded representation of $C_c(\G, \omega)$.  Thus, for any $f \in C_c(\G, \omega)$, 
\begin{align*}
\|q_F(f)\| &= \sup \{ \|\psi\circ q_F(f)\|:\\
& \qquad  \psi\text{ an $I$-norm-bounded representation of }C_c(\G_F, \omega)\} \\
&\leq \sup \{ \|\Psi(f)\|: \Psi\text{ an $I$-norm-bounded representation of }C_c(\G, \omega) \} \\
&= \| f\|.
\end{align*}
Hence, $q_F$  extends to a $*$-homomorphism, also denoted $q_F$, from $C^*(\G, \omega)$ to $C^*(\G|_F ,\omega)$. 

Note that any function $f \in C_c(\G|_F )$ can be extended to $\overline{f} \in C_c(\G)$ by the Tietze Extension Theorem, 
so  that $q_F(\overline{f}) = f$.  Since $C_c(\G|_F ) \subseteq C^*(\G|_F ,\omega)$ is dense, this implies that $q_F:  C^*(\G ,\omega) \to C^*(\G|_F ,\omega)$ is surjective.  In other words,
\[ C^*(\G|_F ,\omega)\cong C^*(\G ,\omega)/\ker q_F.\]

Thus, to see that $C^*(\G|_F ,\omega) =C^*(\G,\omega)_F$, it suffices to show that $\ker q_F = I_F $.  A standard approximation argument will show that $\ker q_F \supseteq I_F$: 
the tricky part is showing that $\ker q_F \subseteq I_F$.  



%
%

To show that $\ker q_F \subseteq I_F$, we will show that  any  representation $L$ of $C_c(\G ,\omega)$ such that  $L(I_F) = 0$ must factor through $q_F$, so that we can write $L = L'\circ q_F$ for some $I$-norm-bounded representation $L'$ of $C_C(\G|_F ,\omega)$.  This will imply that $\ker q_F \subseteq \ker L$ for all such representations $L$, and consequently that $\ker q_F \subseteq I_F$ as desired.



Given an $I$-norm-bounded representation $L: C_c(\G ,\omega)\to B(\mathcal{H})$ such that $L(I_F)= 0$, define $L': C_c(\G|_F ,\omega) \to B(\mathcal{H})$ by 
\[L'(q_F(f)) := L(f).\]
We would like to show that $L'$ is a representation of $C_c(\G|_F ,\omega)$. Note that $L'$ preserves the $*$-algebra structure on $C_c(\G|_F ,\omega)$ because $L$ and $q_F$ do so, being $*$-homomorphisms.  Moreover, $L'$ is nondegenerate because $L$ is and because $q_F: C_c(\G,\omega) \to C_c(\G|_F ,\omega)$ is surjective.  Thus,  we only need to check that $L'$ is well-defined and bounded. 

To see that $L'$ is well defined, we need to show that $L(f) = L(g)$ whenever $q_F(f) = q_F(g)$. 

\begin{lemma}
 If $f, g \in C_c(\G, \omega)$ satisfy $q_F(f) = q_F(g)$, then the function 
 \[h = f -g \in C_c(\G  ,\omega)\]
  lies in $I_F$.  Consequently, $L(f) = L(g)$ and $L'$ is well defined on $C_c(\G|_F,\omega)$.
 \label{well-def}
\end{lemma}
\begin{proof}
Let $\{f_{K,U}\}_{K,U}$ be an approximate unit for $C_0(M\backslash F),$ indexed by pairs $(K, U)$ where $K \subseteq M$ is compact and $U \supseteq F$ is open, such that $U \cup K = M; \ f_{K,U}$ is 1 on $K \backslash U$ and 0 on $F$; and $0 \leq f_{K,U}(m) \leq 1$ for all $m \in M$.  (We can always construct such functions by using Urysohn's Lemma.)

  We will show that the $I$-norm $\| h - f_{K,U} \cdot h\|_I \to 0$, where we take the limit over increasing $K$ and decreasing $U$.   
  Since the norm in $C^*(\G,\omega)$ is dominated by the $I$-norm, it will follow that 
  \[h = \lim_{K,U} f_{K,U}\cdot  h \]
  in $C^*(\G ,\omega)$, and consequently $h\in I_F$.

We first observe that the function $m \mapsto \lambda(\G_m \cap \supp h)$ is bounded, where $\lambda$ denotes our chosen  Haar measure on the fiber group $\G_m \cong G$ (recall that $\G_m \cong \G_n \cong G$ for all $m, n \in M$).  To see that this function is bounded, let $W$ be an open neighborhood of $\supp h$, and use Urysohn's Lemma to construct  $k_W \in C_c(\G )$ such that $k_W|_{\supp h} = 1$ and $\supp k_W \subseteq W$.  Then for any $m \in M$,
  \[\int_{\G_m} k_W(x) d\lambda(x) \geq \lambda(\G_m \cap \supp h).\]
  Moreover, we know from Proposition \ref{cts-int} that 
\[m\mapsto \int_{\G_m} k_W(x) d\lambda(x)\in C_0(M),\]
because $k_W \in C_c(\G)$.
Since this function is an upper bound for the function $m \mapsto \lambda(\G_m \cap \supp h)$, it follows that $m \mapsto \lambda(\G_m \cap \supp h)$ is bounded on $M$, as claimed.
   Let $\ell$ be the maximum value of the function $ m \mapsto \lambda(\G_m \cap \supp h)$.

Let $\epsilon > 0$ be given, and let $U = \pi(h\inv(B_{\epsilon/2\ell}(0)))$.  Then $U \subseteq M$ is open and contains $F$.  Let $K = \pi(\supp h) \subseteq M$; we will show that for any $(K', U') \geq (K,U)$ we have $\| h - f_{K',U'}\cdot h\|_I < \epsilon$.

Recall that  
\begin{align*}
 \| h - f_{K',U'}\cdot h\|_I &= \max \left\{ \sup_{m \in M} \int_{\G_m} | h(x) - f_{K',U'}(m) h(x)| \, d\lambda(x), \right. \\
 & \qquad \left. \sup_{m \in M} \int_{\G_m} | h(x\inv) - f_{K',U'}(m) h(x\inv)| \, d\lambda(x)\right\} .
 \end{align*}
 If $ m \in  K' \backslash U'$, then $f_{K',U'}(m) = 1$ and the above integrals are zero.  If $m \in  U'$, then since $(K', U') \geq (K,U)$ we also have $m \in U$, so 
 \[|h(x)|< \epsilon/2\ell \ \forall \ x \in \G_m.\]
   Moreover, the fact that $0 \leq f_{K',U'}(m) \leq 1 \ \forall \ m \in M$ implies that for any $x \in \G$,
\[|h(x) - f_{K',U'}(\pi(x)) h(x)| \leq 2|h(x)|.\]
It follows that
 \begin{align*}
 \| h - f_{K',U'} \cdot h\|_I &\leq \max \left\{\sup_{m \in U'} \int_{\G_m} | h(x) - f_{K',U'}(m) h(x)| \, d\lambda(x), \right. \\
 & \qquad \left. \sup_{m \in U'} \int_{\G_m} | h(x\inv) - f_{K',U'}(m) h(x\inv)| \, d\lambda(x)\right\} \\
 &\leq \max \left\{ \sup_{m \in U'} \int_{\G_m} 2|h(x)| \, d\lambda(x),\ \sup_{m\in U'} \int_{\G_m} 2|h(x\inv)| \, d\lambda(x) \right\} \\
 & \leq 2 \sup_{m \in U'}\sup_{x \in \G_m} |h(x)| \lambda (\G_m \cap \supp h) \\
 & < 
 \epsilon.
 \end{align*}
Consequently, $\lim_{K,U}\| h - f_{K,U}\cdot h \|_I =0$.  It follows that $h \in I_F$ as claimed, and so $L(h) = 0$.  This proves that $L'$ is well defined.
\end{proof}

Having seen that $L'$ is well defined, we proceed to show that $\| L'(f)\| \leq \|f\|_I$ for any $f \in C_c(\G|_F, \omega)$.  First, we note that Proposition \ref{cts-int} and the definition of the $I$-norm imply that the function $m \mapsto \|q_{\{m\}}(f)\|_I$ is continuous for each $f \in C_c(\G)$.  Consequently, if we fix $f \in C_c(\G,\omega), \ \epsilon > 0$, the set 
\[W_\epsilon = \{m \in M: \| q_{\{m\}}(f)\|_I < \|q_F(f)\|_I + \epsilon\}\]
is open;
 note that $F \subseteq W_\epsilon$. Thus, we can choose $\psi_{f,\epsilon} \in C_0(M)$ such that $0 \leq \psi_{f,\epsilon}(m) \leq 1 \ \forall \ m \in M$; $\psi_{f,\epsilon} = 1$ on $F $; and $\psi_{f,\epsilon} = 0$ off $W_\epsilon$.  Since $\psi_{f,\epsilon} = 1$ on $F$, we have 
\[L(f) = L'(q_F(f)) = L'(q_F(\psi_{f,\epsilon} \cdot f)) = L(\psi_{f,\epsilon} \cdot f).\]
Consequently,
\begin{align*}
\|L'(q_F(f))\| &= \|L(\psi_{f,\epsilon} \cdot f)\| \leq \|\psi_{f,\epsilon} \cdot f\|_I \\
&= \max \left\{\sup_{m\in M} \int_{\G_m} |\psi_{f, \epsilon}(m) f(x)| \, d\lambda(x) ,  \sup_{m\in M} \int_{\G_m} |\psi_{f,\epsilon}(m) f(x\inv)| \, d\lambda(x)  \right\} \\
&\leq \max \left\{ \sup_{m \in W_\epsilon} \int_{\G_m} |f(x)| \, d\lambda(x) , \sup_{m\in W_\epsilon} \int_{\G_m} |f(x\inv)| \, d\lambda(x) \right\}\\
&= \sup_{m \in W_\epsilon} \| q_{\{m\}}(f)\|_I\\
& < \| q_F(f)\|_I + \epsilon.
\end{align*}
Since we can choose such a $\psi_{f,\epsilon}$ for any $\epsilon>0$, it follows that 
\[\|L'(q_F(f))\| \leq \|q_F(f)\|_I.\]
The fact that $q_F: C_c(\G, \omega) \to C_c(\G|_F, \omega)$ is onto now implies that $L'$ is an $I$-norm-bounded representation of $C_c(\G|_F, \omega)$.

 In other words, every representation of $C_c(\G, \omega)$ which kills $I_F$ also factors through $q_F$, so $\ker q_F = I_F$ as claimed.  That is, 
 \[C^*(\G|_F, \omega) \cong C^*(\G, \omega)_F\]
 for any $F \subseteq M$ closed.
This finishes the proof of   Proposition \ref{c0-m}. \end{proof}

Knowing that  $C^*(\G, \omega)_F= C^*(\G|_F ,\omega)$ will be crucial for the arguments in the next section.  
However, we will also need  a  result (Proposition \ref{mv-prop}) about the way  ideals in $C_0(M)$-algebras relate.  Although this result is undoubtedly well-known to experts, we include a proof for completeness.

We begin with an observation about approximate units in $C_0(M)$.  Since $M$ is locally compact Hausdorff, for any closed set $F\subseteq M$ we can write $M \backslash F$ as an increasing union $M\backslash F = \cup_{i\in I} K_i$ of compact sets, and then Urysohn's Lemma tells us that we can find an approximate unit $\{\phi^F_i\}_{i \in I}$ for $C_0(M \backslash F)$ such that $\phi^F_i$ is 1 on $K_i$.  It follows that for any $m \in M\backslash F$, there exists $J \in I$ such that $i \geq J$ implies $\phi^F_i(m) =1$.

\begin{lemma}
Let $A$ be a $C_0(M)$-algebra for a second countable locally compact Hausdorff space $M$, and let $F_1, F_2 \subseteq M$ be closed.  For any $a \in I_{F_1 \cap F_2}$, we can find $g \in I_{F_1}, h \in I_{F_2}$ such that $a-g-h \in I_{F_1 \cup F_2}$.
\label{tech}
\end{lemma}
\begin{proof}
Let $\{\phi^{12}_i\}_{i \in I}$ denote the approximate unit for $C_0(M \backslash (F_1 \cap F_2))$ described above.  Then, given $\epsilon > 0$ and $a \in I_{F_1 \cap F_2}$, there exists $J $ such that  $\| a - \phi^{12}_J\cdot  a\| < \epsilon$.  Let $\{\phi^1_\lambda\}_{\lambda \in \Lambda},  \{\phi^2_\mu \}_{\mu \in S}$ be the analogous approximate units for $C_0(M \backslash F_1)$ and $C_0(M \backslash F_2)$, respectively.  Then 
\[g_\epsilon := (\lim_\lambda \phi^1_\lambda \phi^{12}_J) \cdot a \in I_{F_1}, \quad h_\epsilon := (\lim_\mu \phi^2_\mu\phi_J)\cdot a  \in I_{F_2}.\]
Moreover, $\phi^{12}_J \cdot a - g_\epsilon -h_\epsilon \in I_{F_1 \cup F_2}$.  To see this, fix $m \in F_1 \cup F_2$.  If $m \in F_1 \backslash F_2$, we can choose $L$ large enough that $\mu \geq L$ implies $\phi^2_\mu(m) = 1$; consequently, 
\begin{equation}
\phi^{12}_J(m) - \phi^1_\lambda(m)\phi^{12}_J(m) - \phi^2_\mu(m) \phi^{12}_J(m) = 0
\label{eqn2}
\end{equation}
since $\phi^1_\lambda(m) = 0$ for any $\lambda$ when $m \in F_1$.  Similarly, if $m \in F_2 \backslash F_1$, we can choose $K$ such that $\lambda \geq K$ implies $\phi^1_\lambda(m) = 1$, and so \eqref{eqn2} also holds; and if $m \in F_1 \cap F_2$ then $\phi^{12}_J(m) = 0$ and \eqref{eqn2} still holds.  Thus, if $\lambda, \mu$ are large enough then Equation \eqref{eqn2} holds for all $m \in F_1 \cup F_2$, so taking the limit over $\lambda, \mu$ reveals that 
\[\phi^{12}_J \cdot a - g_\epsilon - h_\epsilon = (\phi^{12}_J - \lim_\lambda \phi^1_\lambda \phi^{12}_J - \lim_\mu \phi^2_\mu \phi^{12}_J) \cdot  a \in I_{F_1 \cup F_2},\] as claimed.

Therefore, denoting the norm in the quotient $A_{F}$ by $\| \cdot \|_{F}$, we see that 
\begin{align*}
\| a - g_\epsilon - h_\epsilon\|_{F_1 \cup F_2} &\leq \| a - \phi^{12}_J \cdot a \|_{F_1 \cup F_2} + \|\phi^{12}_J \cdot a  - g_\epsilon - h_\epsilon\|_{F_1 \cup F_2} \\
&= \| a - \phi^{12}_J \cdot a \|_{F_1 \cup F_2} \\
& \leq \| a - \phi^{12}_J \cdot a\| \\
& < \epsilon.
\end{align*}

Furthermore, since $0 \leq \phi^1_\lambda(m) \leq 1$ for all $m \in M$,
\begin{align*}
\| g_\epsilon - g_{\epsilon'}\| &= \| \lim_\lambda  \phi^1_\lambda(\phi^{12}_J  - \phi^{12}_{J'}) \cdot a\|\\
&\leq \| (\phi^{12}_J- \phi^{12}_{J'}) \cdot a\|\\
&< \epsilon + \epsilon',
\end{align*}
so the net $\{g_\epsilon\}_{\epsilon>0}$ converges in $A$.  The same argument will show that $\{h_\epsilon\}_{\epsilon>0}$ also converges.
Setting $g := \lim_\epsilon g_\epsilon$ and $h := \lim_\epsilon h_\epsilon$, we have $g \in I_{F_1}$ since $g_\epsilon \in I_{F_1} \ \forall \ \epsilon$; similarly, $h \in I_{F_2}$.  We claim that $a-g-h \in I_{F_1 \cup F_2}$.

To see this, let $\delta >0$ be given and suppose $\epsilon$ is small enough that $\| g - g_\epsilon\|, \|h - h_\epsilon\| < \delta$.  Without loss of generality, suppose $\epsilon < \delta$. Then 
\begin{align*}\| a- g-h\|_{F_1 \cup F_2} & \leq \| a- g_\epsilon - h_\epsilon\|_{F_1 \cup F_2} + \|g_\epsilon - g\|_{F_1 \cup F_2} + \| h_\epsilon - h\|_{F_1 \cup F_2}\\
&< \epsilon + 2\delta \\
&< 3\delta.
\end{align*}
It follows that
$a - g - h \in I_{F_1 \cup F_2}$ as claimed.
\end{proof}

\begin{prop}
Let $A$ be a $C_0(X)$-algebra for a locally compact Hausdorff space $X$, and let $F_1, F_2 \subseteq M$ be closed.  Then we have a short exact sequence of $C^*$-algebras 
\begin{equation}
\label{mv}
0 \to A_{F_1 \cup F_2} \to A_{F_1} \oplus A_{F_2} \to A_{F_1 \cap F_2} \to 0.
\end{equation}
\label{mv-prop}
\end{prop}
\begin{proof}
We begin by showing that $I_{F_1 \cup F_2} := \overline{C_0(M \backslash (F_1 \cup F_2)) \cdot A} 
=I_{F_1} \cap I_{F_2}$. 
The containment $I_{F_1 \cup F_2} \subseteq I_{F_1} \cap I_{F_2}$ follows immediately from the definitions; we will prove the other containment.  To that end, suppose $a \in I_{F_1} \cap I_{F_2}$.  Let $\{\phi^1_\lambda\}_\lambda, \{\phi^2_\mu\}_\mu$ be the approximate units for $C_0(M \backslash F_1), C_0(M \backslash F_2)$ respectively that were used in Lemma \ref{tech},  and fix $\epsilon > 0$. 
Then  there exist $\lambda, \mu$ such that  $\|a - \phi^1_\lambda \cdot a\| < \epsilon$ and $\| a - \phi^2_\mu \cdot a\| < \epsilon$; consequently, $\| \phi^1_\lambda \cdot a  - \phi^2_\mu \cdot a\| < 2\epsilon$.

Let $\delta = \epsilon/\|a\|$.  We will now construct $\phi_\epsilon \in C_0(M \backslash (F_1 \cup F_2))$ such that $\| a - \phi_\epsilon \cdot a \| < 4\epsilon$, thus showing that $a \in I_{F_1 \cup F_2}$ as claimed.

 The open set $U = \{m \in M: |\phi^2_\mu(m)| < \delta\}$ contains $F_2$; let $\chi \in C_0(M)$ be a bump function that is 1 on $F_2$ and 0 off $U$.  Then $\phi_\epsilon:=\phi^1_\lambda - \chi \phi^1_\lambda \in C_0(M \backslash (F_1 \cup F_2))$.  Moreover,
 \begin{align*}
 \|a - (\phi^1_\lambda - \chi \phi^1_\lambda) \cdot a\| &\leq \| a - (\phi^1_\lambda - \chi \phi^1_\lambda) \cdot a - (\chi \phi^2_\mu) \cdot a\| + \| (\chi \phi^2_\mu)\cdot  a\| \\
 &\leq \| a - \phi^1_\lambda \cdot a\| + \| \chi\cdot (\phi^1_\lambda \cdot a - \phi^2_\mu\cdot a)\| + \|(\chi \phi^2_\mu ) \cdot a\| \\
 & < \epsilon + 2\epsilon + \delta \| a\| \\
 &= 4\epsilon,
 \end{align*}
since $\chi \phi^2_\mu$ is only nonzero on $U \backslash F_2$, where its maximum modulus is  at most $\delta = \epsilon/\|a\|$.
Since $\epsilon > 0$ was arbitrary, $I_{F_1} \cap I_{F_2} = I_{F_1 \cup F_2}$ as claimed.

Thus, the map $\phi: A_{F_1 \cup F_2} \to A_{F_1} \oplus A_{F_2}$ given by 
\[\phi([a]_{F_1 \cup F_2}) = [a]_{F_1} \oplus [a]_{F_2}\]
is a well-defined, injective $*$-homomorphism.  
Similarly, the map $\psi: A_{F_1} \oplus A_{F_2} \to A_{F_1 \cap F_2}$ given by 
\[\psi( [a]_{F_1} \oplus [b]_{F_2}) = [a-b]_{F_1 \cap F_2}\]
is well-defined and onto, since $I_{F_i} \subseteq I_{F_1 \cap F_2}$ for $i=1,2$.
Since  $\text{Im } \phi \subseteq \ker \psi$ by definition, in order to see that the sequence \eqref{mv} is exact, we merely need to check that $\text{Im } \phi \supseteq \ker \psi$.  The proof of this inclusion relies on Lemma \ref{tech}.

Suppose $[c]_{F_1} \oplus [d]_{F_2} \in \ker \psi$.  Then $c-d \in I_{F_1 \cap F_2}$, so by Lemma \ref{tech}, there exist $g \in I_{F_1}, h \in I_{F_2}$ such that $(c-d) - g - h \in I_{F_1 \cup F_2}$.  In other words, 
\[[c-g]_{F_1 \cup F_2} = [h+d]_{F_1 \cup F_2} =: [b]_{F_1 \cup F_2}.\]
  Since $g \in I_{F_1}, h \in I_{F_2}$, we know $[c-g]_{F_1} = [c]_{F_1}$ and $[h+d]_{F_2} = [d]_{F_2}$.  To sum up, if $\psi([c]_{F_1} \oplus [d]_{F_2}) =  0$, then 
  \[[c]_{F_1} \oplus [d]_{F_2} = [c-g]_{F_1} \oplus [d+h]_{F_2} = [b]_{F_1} \oplus [b]_{F_2} = \phi([b]_{F_1 \cup F_2}),\]
  and $\ker \psi \subseteq \text{Im } \phi$ as claimed.
In other words, the sequence \eqref{mv} is  exact.
 \end{proof}

\section{Mayer-Vietoris}
In this section, we will translate the results about $C_0(M)$-algebras obtained in the previous section into statements about the $K$-theory groups $K_*(C^*(\G, \omega))$ of the $C^*$-algebras of twisted group bundles.


$K$-theory (cf. \cite{wegge-olsen, rordam, blackadar}) is a covariant,  $\Z_2$-graded homotopy-invariant functor from the category of $C^*$-algebras to the category of abelian groups.  In plain language, this means that $K$-theory associates to each $C^*$-algebra $A$ a pair of abelian groups, $K_0(A)$ and $K_1(A)$.  The $K$-theory groups are constructed from equivalence classes of projections  in certain $C^*$-algebras associated to $A$, and $*$-homomorphisms $A \to B$ of $C^*$-algebras induce homomorphisms $K_*(A) \to K_*(B)$ in such a way that homotopic $*$-homomorphisms induce the same map on $K$-theory.

Among the many useful properties of $K$-theory is the so-called ``continuity of $K$-theory'' (cf.~\cite{wegge-olsen} Proposition 6.2.9), which implies 
that \begin{equation}
\label{cont-kthy}
K_*(\oplus_{n \in \N} A_n) = \oplus_{n \in \N} K_*(A_n).\end{equation}
Also relevant to our discussion in this article  is the 6-term exact sequence in $K$-theory (cf.~\cite{blackadar} Theorem 9.3.1, \cite{rordam} Theorem 12.1.2): 
any short exact sequence of $C^*$-algebras 
\[0 \to J \to A \to B \to 0\]
induces a 6-term exact sequence of the $K$-groups 
\begin{equation*} 
 \begin{tikzpicture}
\node (F-0) at (-4,2){$K_0(J)$};
\node (F1+F2-0) at (0,2){$K_0(A)$};
\node (F1F2-0) at (4,2){$K_0(B)$};
\node (F-1) at (4,0){$K_1(J)$};
\node(F1+F2-1) at (0,0){$K_1(A)$};
\node(F1F2-1) at (-4,0){$K_1(B)$};
\begin{scope}[->,  line width=0.7]
			\draw  (F-0) to  (F1+F2-0); 
			\draw (F1+F2-0) to  (F1F2-0);
			\draw (F1F2-0) to (F-1);
			\draw (F-1) to (F1+F2-1);
			\draw (F1+F2-1) to (F1F2-1);
			\draw (F1F2-1) to (F-0);
\end{scope}
\end{tikzpicture} 
\end{equation*}
Thus, the short exact sequence \eqref{mv} gives rise to the following 6-term exact sequence in $K$-theory: \begin{equation} 
\label{kthy-mv}
 \begin{tikzpicture}
\node (F-0) at (-4,2){$K_0(A_{F_1 \cup F_2})$};
\node (F1+F2-0) at (0,2){$K_0(A_{F_1}) \oplus K_0(A_{F_2})$};
\node (F1F2-0) at (4,2){$K_0(A_{F_1 \cap F_2})$};
\node (F-1) at (4,0){$K_1(A_{F_1 \cup F_2})$};
\node(F1+F2-1) at (0,0){$K_1(A_{F_1}) \oplus K_1(A_{F_2})$};
\node(F1F2-1) at (-4,0){$K_1(A_{F_1 \cap F_2})$};
\begin{scope}[->,  line width=0.7]
			\draw  (F-0) to  (F1+F2-0); 
			\draw (F1+F2-0) to  (F1F2-0);
			\draw (F1F2-0) to (F-1);
			\draw (F-1) to (F1+F2-1);
			\draw (F1+F2-1) to (F1F2-1);
			\draw (F1F2-1) to (F-0);
\end{scope}
\end{tikzpicture} 
\end{equation}
Since $C^*(\G, \omega)$ is a $C_0(M)$-algebra whenever $\pi: \G \to M$ is a group bundle, we propose to use this diagram to study the $K$-theory groups associated to a homotopy of 2-cocycles on $\G$.  The following definition is a special case of \cite{transf-gps} Definition 2.11.

Given a group bundle $\pi: \G \to M$, we can construct  the associated group bundle $\tilde{\pi}: \G \times [0,1] \to M \times [0,1]$, which has total space $ \G \times [0,1]$, and fiber $\pi\inv(m)$ over $(m,t) \in M \times [0,1]$ for any $t \in [0,1]$.

\begin{defn}
A \emph{homotopy of 2-cocycles} on a group bundle $\pi: \G \to M$ is a 2-cocycle $\Omega$ on the group bundle $\tilde{\pi}: \G \times [0,1] \to M \times [0,1]$.
\end{defn}

  Observe that a homotopy $\Omega$ of 2-cocycles gives rise to a family $\{\omega_t\}_{t\in[0,1]}$ of 2-cocycles on the original group bundle $ \G \to M$, which varies continuously in $t$ thanks to the continuity of $\Omega$.
\begin{example}
Suppose that $\pi: V \to M$ is a vector bundle and that $\sigma: V\2 \to \R$ is a 2-form on $V$.  The function $\Omega: V\2 \times [0,1] \to \T$ given by  
\[ \Omega((v,w,t)) = \exp^{2\pi i t \sigma(v,w)}\]
is a homotopy of 2-cocycles on $V$, with 
\[\omega_0(v,w) = 1; \qquad \omega_1(v,w) = \exp^{2\pi i \sigma(v,w)}.\]
\end{example}

If $\Omega = \{\omega_t\}_{t\in[0,1]}$ is a homotopy of 2-cocycles on $\G$, we have a natural $*$-homomorphism 
\[Q_t: C^*(\G \times [0,1], \Omega) \to C^*(\G, \omega_t)\]
for any $t \in [0,1]$, which is given on the dense subalgebra $C_c(\G\times [0,1])$ by evaluation at $t$.  Observe that if $F\subseteq M$ is closed, then 
\[Q_t \circ q_{F \times [0,1]} = q_F \circ Q_t,\]
since this equality evidently holds on the dense subalgebra $C_c(\G \times [0,1])$, and hence holds in general.
Consequently, the diagrams \eqref{kthy-mv} for the algebras 
\[ A = C^*(\G \times [0,1], \Omega), A^t= C^*(\G, \omega_t)\]
 can be connected into a larger commutative diagram:
\begin{equation} 
\label{big-kthy-mv}
 \begin{tikzpicture}
\node (F-0) at (-5,4){$K_0(A_{F_1 \cup F_2})$};
\node (F1+F2-0) at (0,4){$K_0(A_{F_1}) \oplus K_0(A_{F_2})$};
\node (F1F2-0) at (5,4){$K_0(A_{F_1 \cap F_2})$};
\node (F-1) at (5,0){$K_1(A_{F_1 \cup F_2})$};
\node(F1+F2-1) at (0,0){$K_1(A_{F_1}) \oplus K_1(A_{F_2})$};
\node(F1F2-1) at (-5,0){$K_1(A_{F_1 \cap F_2})$};
\begin{scope}[->,  line width=0.7]
			\draw  (F-0) to  (F1+F2-0); 
			\draw (F1+F2-0) to  (F1F2-0);
			\draw (F1F2-0) to (F-1);
			\draw (F-1) to (F1+F2-1);
			\draw (F1+F2-1) to (F1F2-1);
			\draw (F1F2-1) to (F-0);
\end{scope}

\node (F-0t) at (-3,3){$K_0(A^t_{F_1 \cup F_2})$};
\node (F1+F2-0t) at (0,3){$K_0(A^t_{F_1}) \oplus K_0(A^t_{F_2})$};
\node (F1F2-0t) at (3,3){$K_0(A^t_{F_1 \cap F_2})$};
\node (F-1t) at (3,1){$K_1(A^t_{F_1 \cup F_2})$};
\node(F1+F2-1t) at (0,1){$K_1(A^t_{F_1}) \oplus K_1(A^t_{F_2})$};
\node(F1F2-1t) at (-3,1){$K_1(A^t_{F_1 \cap F_2})$};
\begin{scope}[->,  line width=0.7]
			\draw  (F-0t) to  (F1+F2-0t); 
			\draw (F1+F2-0t) to  (F1F2-0t);
			\draw (F1F2-0t) to (F-1t);
			\draw (F-1t) to (F1+F2-1t);
			\draw (F1+F2-1t) to (F1F2-1t);
			\draw (F1F2-1t) to (F-0t);
			\draw (F-0) to (F-0t);
			\draw (F1+F2-0) to (F1+F2-0t);
			\draw (F1F2-0) to (F1F2-0t);
			\draw (F-1) to (F-1t);
			\draw (F1+F2-1) to (F1+F2-1t);
			\draw (F1F2-1) to (F1F2-1t);
\end{scope}
\end{tikzpicture} 
\end{equation}
where all of the arrows connecting the inner and outer diagrams arise from the map $Q_t$.

\begin{thm} 
\label{main}
Let $\G \to M$ be a second countable, locally trivial, amenable group bundle, with $\Omega=\{\omega_t\}_{t\in[0,1]}$ a homotopy of 2-cocycles on $\G$.  Then 
\[K_*(C^*(\G, \omega_t)) \cong K_*(C^*(\G \times [0,1], \Omega))\]
for any $t \in [0,1]$.
\end{thm}
\begin{proof}
We begin by considering the case when $M$ is compact.

For each $m \in M$, let $V_m$ be a compact neighborhood of $m$ such that $\G$ trivializes over $V_m$.  Then $\G$ also trivializes over $V_m \cap V_n$.  In other words, $\G|_{V_m}$ and $\G|_{V_m \cap V_n}$ are transformation groups over compact spaces (with the trivial action of the group  $\G_m \cong G$ on the spaces $V_m, V_n, V_m \cap V_n \subseteq M$).  By hypothesis, $\G|_{V_m}$ is a bundle of amenable groups, and so Theorem 3.5
 of \cite{renault-2013} tells us that $\G|_{V_m}$ is an amenable groupoid; in other words,
\[A_{V_m} = C^*(\G|_{V_m} \times [0,1], \Omega) \cong C^*_r(\G|_{V_m} \times [0,1], \Omega).\]
Theorem 5.1 of \cite{transf-gps} states that a homotopy $\Omega = \{\omega_t\}_{t\in[0,1]}$ of 2-cocycles on a second countable locally compact transformation group $G \ltimes X$ induces an isomorphism 
\[(Q_t)_*: K_*(C^*_r(G \ltimes X \times [0,1], \Omega)) \to K_*(C^*_r(G \ltimes X, \omega_t))\]
for any $t \in [0,1]$, as long as  $G$ satisfies the Baum-Connes conjecture with coefficients.  Applying this result to the case $G = \G_m$ and $X = V_m$, since amenable groups satisfy the Baum-Connes conjecture with coefficients,
we see that $Q_t: A_{V_m} \to A^t_{V_m}$ induces an isomorphism on $K$-theory, and that the same is true for $Q_t: A_{V_n} \to A^t_{V_n}$ and $Q_t: A_{V_m \cap V_n} \to A^t_{V_m \cap V_n}$.  The Five Lemma now implies that $Q_t: A_{V_m \cup V_n} \to A^t_{V_m \cup V_n}$ induces a $K$-theoretic equivalence.

Iterating this procedure, since the compactness of $M$ implies the existence of a finite cover of $M$ consisting of the sets $V_m$, we see that $Q_t$ induces an isomorphism 
\[(Q_t)_*: K_*(C^*(\G \times [0,1], \Omega)) \to K_*(C^*(\G, \omega_t)),\]
as claimed.

For the general case, when $M$ is not compact, write $M = \cup_{i=1}^\infty U_i$, where $U_i \subseteq U_{i+1} $ and $\overline{U_i}$ is compact for all $i$.  (We are indebted to Nigel Higson for suggesting this argument.) Then $M= F_1 \cup F_2$, where 
\[F_1 = \cup_{i=0}^\infty \overline{U_{2i+1}} \backslash U_{2i}; \qquad F_2 = \cup_{i=1}^\infty \overline{U_{2i}} \backslash U_{2i -1}.\]
Without loss of generality, we may assume that $\partial U_i \cap \partial U_{i+1} = \emptyset \ \forall \ i$;
then $F_1, F_2$, and $F_1 \cap F_2 = \cup_{i=1}^\infty \partial U_i$ are each closed sets, consisting of countably many disjoint compact sets.

Consequently, for $F = F_1, F_2, F_1 \cap F_2$, we see that $C^*(\G|_F \times [0,1], \Omega)$ and $C^*(\G|_F, \omega_t)$ both break up as a countable direct sum 
\[C^*(\G|_F \times [0,1],\Omega) \cong \bigoplus_{n=0}^\infty C^*(\G|_{F^n} \times [0,1],\Omega); \qquad C^*(\G|_F, \omega_t) \cong \bigoplus_{n=0}^\infty C^*(\G|_{F^n}, \omega_t)\]
where $F^n$ is compact for all $n$.  Since we established above that for a compact set $F^n$,  
\[(Q_t)_*: K_*(C^*(\G|_{F^n} \times [0,1], \Omega)) \to K_*(C^*(\G|_{F^n}, \omega_t))\]
is an isomorphism for all $t \in [0,1]$, it follows from \eqref{cont-kthy} that 
\[(Q_t)_*: K_*(C^*(\G|_{F} \times [0,1], \Omega)) \to K_*(C^*(\G|_{F}, \omega_t))\]
is also an isomorphism for $F = F_1, F_2, F_1 \cap F_2$.  Since $M = F_1 \cup F_2$, the short exact sequence of \eqref{mv} combines with the Five Lemma (following the same argument given above in the case $M$ is compact) to tell us that 
\[(Q_t)_*: K_*(C^*(\G \times [0,1], \Omega)) \to K_*(C^*(\G, \omega_t))\]
is also an isomorphism.  This finishes the proof of Theorem \ref{main}.
\end{proof}

Our main result now follows immediately:

\begin{cor}
\label{main-cor}
Any homotopy $\Omega = \{\omega_t\}_{t \in [0,1]}$ of 2-cocycles on a second countable, locally trivial amenable group bundle $\G \to M $ induces an isomorphism 
\[K_*(C^*(\G, \omega_0)) \cong K_*(C^*(\G, \omega_1)).\]
\end{cor}

When we consider the particular case when $V \to M$ is a vector bundle, we obtain the following generalization of Theorem 1 from \cite{plymen-weyl-bdl}:
\begin{cor}
Let $V \to M$ be a vector bundle, and let $\sigma: V\2 \to \R$ be a bilinear 2-form on $V$.  Setting $\omega (v,w) = e^{2\pi i \sigma(v,w)}$, we have 
\[K_*(C^*(V, \omega)) \cong K_*(C^*(V)) = K_*(C_0(V^*)).\]
In particular, if $V$ is even dimensional, then 
\[K_*(C^*(V, \omega)) \cong K_*(C_0(M)).\]
\end{cor}

\bibliographystyle{amsplain}
\bibliography{eagbib}

\providecommand{\bysame}{\leavevmode\hbox to3em{\hrulefill}\thinspace}
\providecommand{\MR}{\relax\ifhmode\unskip\space\fi MR }
\providecommand{\MRhref}[2]{%
  \href{http://www.ams.org/mathscinet-getitem?mr=#1}{#2}
}
\providecommand{\href}[2]{#2}
\begin{thebibliography}{10}

\bibitem{blackadar}
B.~Blackadar, \emph{{$K$}-theory for operator algebras}, Cambridge University
  Press, 1998.

\bibitem{jon-astrid}
J.H. Brown and A.~an~Huef, \emph{Decomposing the {$C^*$}-algebras of groupoid
  extensions}, Proceedings of the American Mathematical Society \textbf{142}
  (2014), 1261--1274.

\bibitem{clark-aH}
L.O. Clark and A.~an~Huef, \emph{The representation theory of ${C}^*$-algebras
  associated to groupoids}, Mathematical Proceedings of the Cambridge
  Philosophical Society \textbf{153} (2012), 167--191.

\bibitem{eag-kgraph}
E.~Gillaspy, \emph{${K}$-theory and homotopies of 2-cocycles on higher-rank
  graphs}, arXiv:1403.3799 (2014).

\bibitem{transf-gps}
\bysame, \emph{${K}$-theory and homotopies of 2-cocycles on transformation
  groups}, Journal of Operator Theory (to appear).

\bibitem{goehle-mmII}
G.~Goehle, \emph{The {M}ackey machine for crossed products by regular
  groupoids. {II}}, Rocky Mountain Journal of Mathematics \textbf{42} (2012),
  873--900.

\bibitem{dd-fell}
A.~an Huef, A.~Kumjian, and A.~Sims, \emph{A {D}ixmier-{D}ouady theorem for
  {F}ell algebras}, Journal of Functional Analysis \textbf{260} (2011),
  1543--1581.

\bibitem{c*-diagonals}
A.~Kumjian, \emph{On ${C}^*$-diagonals}, Canadian Journal of Mathematics
  \textbf{38} (1986), 969--1008.

\bibitem{cts-trace-gpoid-III}
P.S. Muhly, J.N. Renault, and D.P. Williams, \emph{Continuous-trace groupoid
  ${C}^*$-algebras. {III}}, Transactions of the American Mathematical Society
  \textbf{348} (1996), 3621--3641.

\bibitem{cts-trace-gpoid-II}
P.S. Muhly and D.P. Williams, \emph{Continuous trace groupoid ${C}^*$-algebras
  {II}}, Mathematica Scandinavica \textbf{70} (1992), 127--145.

\bibitem{plymen-weyl-bdl}
R.J. Plymen, \emph{The {W}eyl bundle}, Journal of Functional Analysis
  \textbf{49} (1982), 186--197.

\bibitem{rordam}
M.~R\o rdam, F.~Larsen, and N.J. Laustsen, \emph{An introduction to
  {$K$}-theory for ${C}^*$-algebras}, London Mathematical Society Student
  Texts, no.~49, Cambridge University Press, 2000.

\bibitem{renault}
J.~Renault, \emph{A groupoid approach to ${C}^*$-algebras}, Lecture Notes in
  Mathematics, vol. 793, Springer-Verlag, 1980.

\bibitem{renault-2013}
\bysame, \emph{Topological amenability is a {B}orel property}, arXiv:1302:0636
  (2013).

\bibitem{TXLG}
J.-L. Tu, P.~Xu, and C.~Laurent-Gengoux, \emph{Twisted ${K}$-theory of
  differentiable stacks}, Annales Scientifiques de l'\'Ecole Normale
  Sup\'erieure \textbf{37} (2004), 841--910.

\bibitem{wegge-olsen}
N.E. Wegge-Olsen, \emph{${K}$-theory and ${C}^*$-algebras: A friendly
  approach}, Oxford University Press, 1993.

\bibitem{xprod}
D.P. Williams, \emph{Crossed products of ${C}^*$-algebras}, Mathematical
  Surveys \& Monographs, vol. 134, AMS, 2007.

\end{thebibliography}
\end{document}